\documentclass[11pt,a4paper]{amsart}

\usepackage{amsmath}
\usepackage{amssymb}
\usepackage{graphicx}
\usepackage{color}
\usepackage{url}
\usepackage{hyperref}
\hypersetup{
    colorlinks,
    citecolor=black,
    filecolor=black,
    linkcolor=black,
    urlcolor=black
}

\newtheorem{thm}{Theorem}
\newtheorem*{que*}{Question}
\newtheorem{prop}[thm]{Proposition}
\newtheorem{lem}[thm]{Lemma}

\makeatletter
\newcommand*{\house}[1]{%
  \mathord{%
    \mathpalette\@house{#1}%
  }%
}
\newcommand*{\@house}[2]{%
  \dimen@=\fontdimen8 %
      \ifx#1\scriptscriptstyle\scriptscriptfont
      \else\ifx#1\scriptstyle\scriptfont
      \else\textfont\fi\fi
      3 %
  \sbox0{%
    $#1%
      \vrule width\dimen@\relax
      \overline{%
        \kern2\dimen@
        \begingroup 
          #2%
        \endgroup
        \kern2\dimen@
      }%
      \vrule width\dimen@\relax
      \mathsurround=1.5\dimen@ 
    $%
  }%
  \ht0=\dimexpr\ht0-\dimen@\relax
  \dp0=\dimexpr\dp0+2\dimen@\relax
  \vbox{%
    \kern\dimen@ 
    \copy0 %
  }%
}

\begin{document}
\title[Arithmetic and geometry of skew-reciprocal Polynomials]{On the arithmetic and the geometry of skew-reciprocal~polynomials}
\author{Livio Liechti}
\thanks{The author was supported by the Swiss National Science Foundation (grant nr.\ 175260)}
\address{Department of Mathematics, University of Fribourg, Ch.\ du Mus\'ee 23, 1700 Fribourg, Switzerland}
\email{livio.liechti@unifr.ch}

\begin{abstract} 
We reformulate Lehmer's question from 1933 and a question due to Schinzel and Zassenhaus 
from 1965 in terms of a comparison of the Mahler measures and the houses, respectively,
of monic integer reciprocal and skew-reciprocal polynomials of the same degree. 
This entails that understanding the difference between orientation-preserving and orientation-reversing 
mapping classes is at least as complicated as answering these questions.
\end{abstract}

\maketitle

\section{Introduction}
\noindent
Kronecker's theorem from 1857 states that if a nonzero algebraic integer 
is not a root of unity, then it has a Galois conjugate outside the unit circle~\cite{Kronecker}. 
Roughly speaking, Lehmer in 1933 and Schinzel and Zassenhaus in 1965 asked whether 
this statement can be made quantitatively precise using the Mahler measure or the house of polynomials,
respectively~\cite{Lehmer, SZ}.
The explicit questions they rose 
remain unanswered to this day.\medskip

\noindent
Our main results (Theorems~\ref{Lehmerthm} and~\ref{SZthm}) provide reformulations of Lehmer's question
and the question of Schinzel and Zassenhaus in terms of a comparison of reciprocal and 
skew-reciprocal polynomials. A polynomial~$f\in\mathbf{Z}[t]$ of degree~$d$ is~\emph{reciprocal} if~$f(t) = t^df(t^{-1}).$
Similarly, a polynomial~$f\in\mathbf{Z}[t]$ of even degree~$2d$ is~\emph{skew-reciprocal} 
if~$f(t)=(-1)^dt^{2d}f(-t^{-1}).$\medskip

\noindent
As we discuss in Section~\ref{mappingclasses}, reciprocal and skew-reciprocal polynomials arise naturally as the 
characteristic polynomials of integer symplectic and anti-symplectic matrices, respectively. These are in turn exactly 
the actions induced on the first homology of closed surfaces by orientation-preserving and orientation-reversing 
mapping classes, respectively. Via these equivalences, Theorems~\ref{Lehmerthm} and~\ref{SZthm} are our 
way of making precise the following statement:
\emph{Understanding the difference between orientation-preserving and orientation-reversing mapping classes is at 
least as complicated as answering Lehmer's question and the question of Schinzel and Zassenhaus}. 

\subsection{The Mahler measure}
Let~$f\in\mathbf{Z}[t]$ be a monic polynomial. The \emph{Mahler measure}~$M(f)$ of~$f$ is 
the modulus of the product of all zeroes of~$f$ outside the unit circle, counted with multiplicity:
~$$M(f)=\prod_{f(\alpha)=0}\mathrm{max}(1,|\alpha|).$$

\begin{que*}[Lehmer's question~\cite{Lehmer}]
\label{weaklehmerque}
Does the set of Mahler measures of monic integer polynomials accumulate at~$1$?
\end{que*}

\noindent
The smallest Mahler measure larger than~$1$ found by Lehmer is the one of the polynomial 
~$$L(t) = t^{10}+t^9-t^7-t^6-t^5-t^4-t^3+t+1.$$
We call~$\lambda_L = M(L) \approx 1.17628$ \emph{Lehmer's number}.
It is still the smallest known Mahler measure larger than~$1$.

\begin{que*}[Lehmer's question, strong version~\cite{Lehmer}]
\label{stronglehmerque}
Is Lehmer's number the smallest Mahler measure larger than~$1$ among all monic integer polynomials?
\end{que*}

\noindent
Our first result provides a reformulation of 
Lehmer's question.
Let~$R_i$ be the smallest Mahler measure larger than~$1$
among monic integer reciprocal polynomials of degree~$2^i$,
and let~$S_i$ be the smallest Mahler measure larger than~$1$ among monic integer skew-reciprocal 
polynomials of degree~$2^i$. 

\begin{thm}
\label{Lehmerthm}
Lehmer's number~$\lambda_L$
is the smallest Mahler measure larger than~$1$ among 
all monic integer polynomials exactly if~$R_i = S_i$ for all~$i\ge5$. 
Furthermore, the set of Mahler measures of monic integer polynomials accumulates at~$1$ 
exactly if~$\prod_{i=5}^N \frac{R_i}{S_i}$ converges to $\lambda_L^{-1}$ as~$N\to\infty$. 
\end{thm}

\noindent
Theorem~\ref{Lehmerthm} gives a reformulation of Lehmer's question in terms of a comparison of orientation-preserving and orientation-reversing mapping classes.
Indeed, from our discussion in Section~\ref{mappingclasses} it follows that~$R_i$ and~$S_i$ are the smallest Mahler measures larger than~$1$
among all characteristic polynomials of actions induced on the first homology of the closed surface of genus~$2^{i-1}$ by orientation-preserving and orientation-reversing 
mapping classes, respectively, compare with Proposition~\ref{polycharacter}. 

\subsection{The house}
Let~$f\in\mathbf{Z}[t]$ be a monic polynomial. 
The \emph{house}~$\house{f}$ of~$f$ is the largest modulus among the zeroes of~$f$,
$$\house{f} = \mathrm{max}_{f(\alpha)=0}|\alpha|.$$

\begin{que*}[Schinzel--Zassenhaus~\cite{SZ}]
Does there exists a universal constant~$c>0$ so that any house larger than~$1$
of an irreducible monic integer polynomial is at least~$1+\frac{c}{d},$
where~$d$ is the degree of the polynomial?
\end{que*}

\noindent
Our second result is a reformulation of the question of Schinzel and Zassenhaus in terms of a comparison of reciprocal 
and skew-reciprocal polynomials. Let~$\lambda_{i}$ and~$\widetilde\lambda_{i}$ be the smallest 
houses larger than~$1$ among all monic integer reciprocal and 
skew-reciprocal polynomials of degree~$2^i$, respectively.
Let~$r_{i} = 2^i\log(\lambda_{i})$ and~$s_{i} = 2^i\log(\widetilde\lambda_{i})$. 

\begin{thm}
\label{SZthm}
There exists a universal constant~$c>0$ so that any house larger than~$1$
of an irreducible monic integer polynomial is at least~$1+\frac{c}{d},$ 
where~$d$ is the degree of the polynomial,
exactly if the set~$\left\{\frac{q_N}{q_n}\in\mathbf{R} : n,N\in\mathbf{N}, n<N \right\}\subset\mathbf{R}$
is bounded away from zero, where~$q_m = \prod_{i=1}^{m}\frac{r_{i}}{s_{i}}.$
\end{thm}

\noindent
Theorem~\ref{SZthm} gives a reformulation of the 
question of Schinzel and Zassenhaus in terms of a comparison of orientation-preserving and orientation-reversing 
mapping classes. Again, from our discussion in Section~\ref{mappingclasses} 
it follows that~$\lambda_i$ and~$\widetilde\lambda_i$ are 
equal to~$\delta_{2^{i-1}}^\mathrm{hom}$ and~$\widetilde\delta_{2^{i-1}}^\mathrm{hom}$, respectively, compare with Proposition~\ref{polycharacter}.
Here,~$\delta_g^\mathrm{hom}$ and~$\widetilde\delta_g^\mathrm{hom}$ denote the minimal
spectral radii larger than~$1$ among actions induced on the first homology of the closed surface of genus~$g$ 
by orientation-preserving and orientation-reversing mapping classes, respectively.  
\medskip

\noindent
The easiest way to fulfil the second statement in Theorem~\ref{SZthm} is if for all but finitely many~$i$, 
we have~$r_i\ge s_i$ and hence~$\delta_{2^{i-1}}^\mathrm{hom}\ge\widetilde\delta_{2^{i-1}}^\mathrm{hom}$.   
Combining the results and the conjectures of Hironaka~\cite{Hironaka}, Lanneau and Thiffeault~\cite{LT},
and Strenner and the author~\cite{LS},
this statement seems to have a chance of being true, 
at least when restricting to the actions induced by pseudo-Anosov 
mapping classes with an orientable invariant foliation. \medskip

\noindent
Finally, we reformulate the question of Schinzel and Zassenhaus as a comparison of~$\delta_g^\mathrm{hom}$ and
the minimal dilatation~$\delta_g$ among pseudo-Anosov mapping classes on the closed surface of genus~$g$ (defined in Section~\ref{pAsection}). 

\begin{thm}
\label{dilthm}
There exists a universal constant~$c>0$ so that any house larger than~$1$
of an irreducible monic integer polynomial is at least~$1+\frac{c}{d},$ 
where~$d$ is the degree of the polynomial, exactly if there exists a universal constant~$C>0$ 
so that for all~$g$, $(\delta_g^\mathrm{hom})^C\ge\delta_g.$
\end{thm}

\subsection{Proof strategy}
\noindent
Lehmer's question and hence the question of Schinzel and Zassenhaus is solved in the case of
irreducible nonreciprocal polynomials, due to a result of Breusch~\cite{Breusch}.

\begin{thm}[Breusch~\cite{Breusch}]
\label{breusch}
The Mahler measure of any integer nonreciprocal irreducible polynomial other than~$(t-1)$ and~$t$ is greater than~$1.179$.
\end{thm}

\noindent
The constant of the bound in Theorem~\ref{breusch} is not optimal, but it suffices for our purpose.
For the optimal constant and more results on the Mahler measure and the house 
of integer polynomials, see Smyth's survey~\cite{Smyth}.  
\medskip

\noindent
We use Theorem~\ref{breusch} in order to reduce Lehmer's question and the question
of Schinzel and Zassenhaus to the case of irreducible reciprocal polynomials. 
From there, the main insight for the proofs of Theorems~\ref{Lehmerthm} and~\ref{SZthm} consists of the fact that one can in a controlled way compare 
skew-reciprocal polynomials of degree~$2^{i+1}$ with reciprocal polynomials of degree~$2^i$.\medskip

\noindent
Theorem~\ref{dilthm} follows rather directly from the fact that the minimal dilatations~$\delta_g$ 
among pseudo-Anosov mapping classes satisfy an inequality as in the question by Schinzel and Zassenhaus. 
This is a result due to Penner~\cite{Pe2}.

\begin{thm}[Penner~\cite{Pe2}]
\label{Penner}
There exist universal constants~$R, R'>1$ so that $$R\le(\delta_g)^g\le R'.$$
\end{thm}

\subsection{Organisation}
We prove Theorems~\ref{Lehmerthm} and~\ref{SZthm} in Section~2.
In Section~3, we relate reciprocal and skew-reciprocal polynomials with the 
characteristic polynomials of the actions induced on the first homology of closed surfaces by
mapping classes. We finally prove Theorem~\ref{dilthm}. \medskip

\noindent
\emph{Acknowledgements:} I would like to thank S.\ Baader, E.\ Hironaka, C.\ McMullen, B.\ Strenner 
and an anonymous referee for helpful discussions and comments.


\section{Reciprocal vs.\ skew-reciprocal polynomials}

\noindent
Recall that a polynomial~$f\in\mathbf{Z}[t]$ of even degree~$2d$ is called {reciprocal} if we have~$f(t)=t^{2d}f(t^{-1})$,
and {skew-reciprocal} if~$f(t)=(-1)^{d}t^{2d}f(-t^{-1})$.

\begin{lem}
\label{rekurs}
Let~$f\in\mathbf{Z}[t]$ be a monic skew-reciprocal polynomial of degree~$2^{i+1}$ with~$\house{f}>1$. 
Then either~$f(t)=g(t^2)$, where~$g(t)$ is a reciprocal polynomial of degree~$2^i$, or~$f$ has a nonreciprocal irreducible factor other than~$(t-1)$. 
\end{lem}

\begin{proof}
Let~$f\in\mathbf{Z}[t]$ be a monic skew-reciprocal polynomial of degree~$2^{i+1}$ such that~$\house{f}>1$. \smallskip

\noindent
\emph{Case 1:~$f$ is reciprocal}. If a polynomial~$f\in\mathbf{Z}[t]$ of degree~$2^{i+1}$ is both reciprocal and skew-reciprocal, 
we have~$f(t)=g(t^2)$, where~$g(t)$ is a reciprocal polynomial of degree~$2^{i}$.\smallskip

\noindent
\emph{Case 2:~$f$ is not reciprocal}. If~$f$ is not reciprocal, it must have at least one nonreciprocal irreducible factor. 
Moreover,~$(t-1)$ cannot be the only nonreciprocal irreducible factor.  
Indeed, if~$(t-1)$ was the only nonreciprocal irreducible factor, 
then it would have to appear to an even power, since the constant coefficient of~$f$ is~$+1$.
This follows directly from the definition of skew-reciprocity and the degree of~$f$ being divisible by four.
However, an even power of~$(t-1)$ is reciprocal and hence so would be the polynomial~$f$, a contradiction. 
We have shown that the polynomial~$f$ must contain a nonreciprocal irreducible factor other than~$(t-1)$. 
\end{proof}

\subsection{Mahler measures}
Recall that the {Mahler measure}~$M(f)$ of a monic polynomial~$f\in\mathbf{Z}[t]$ is the modulus of the product of all 
zeroes of~$f$ outside the unit circle, counted with multiplicity:~$$M(f)=\prod_{f(\alpha)=0}\mathrm{max}(1,|\alpha|).$$
We remark that~$M(f(t))=M(f(t^2))$ for any polynomial~$f\in\mathbf{Z}[t]$. 
Let~$R_i$ be the smallest Mahler measure larger than~$1$
among monic integer reciprocal polynomials of degree~$2^i$, 
and let~$S_i$ be the smallest Mahler measure larger than~$1$ among monic integer skew-reciprocal 
polynomials of degree~$2^i$. 

\begin{lem}
\label{Mahlerequality}
$S_{i+1} = R_i$ for~$i\ge 4$. 
\end{lem}

\begin{proof}
Let~$\lambda_L\approx 1.17628$ be Lehmer's number, an algebraic integer of degree~$10$.
We have~$R_i \le \lambda_L$ for~$i\ge4$: the minimal polynomial of~$\lambda_L$ is reciprocal, so we can multiply
it with a power of~$(t+1)$ to obtain a reciprocal polynomial of arbitrary degree and Mahler measure equal to~$\lambda_L$. 
Furthermore, we have~$S_{i+1}\le R_i\le \lambda_L$ for~$i\ge4$. Indeed, if~$g$ is a reciprocal polynomial of 
degree~$2^i$, then~$f(t)=g(t^2)$ is a skew-reciprocal polynomial of degree~$2^{i+1}$ and~$M(f)=M(g)$. \medskip

\noindent
In order to prove~$S_{i+1}\ge R_i$ for~$i\ge4$, let~$f(t)$ be a monic skew-reciprocal polynomial of degree~$2^{i+1}$ and
of Mahler measure~$>1$. By Lemma~\ref{rekurs},~$f(t)$ either has an irreducible non-reciprocal factor other than~$(t-1)$, 
or equals~$g(t^2)$ for some reciprocal polynomial~$g(t)$. 
In the former case, Theorem~\ref{breusch} implies~$M(f)\ge \lambda_L\ge R_i$. 
In the latter case, we have~$M(f)=M(g)\ge R_i$. 
\end{proof}

\begin{proof}[Proof of Theorem~\ref{Lehmerthm}]
By Theorem~\ref{breusch}, Lehmer's question can be reduced to monic irreducible reciprocal polynomials. 
By multiplication with factors~$(t+1)$, one sees that Lehmer's question is in turn equivalent to the same question for 
(not necessarily irreducible) monic reciprocal polynomials of some degree~$2^i$, that is, for~$R_i$. 
Now, some number~$R_N$ is smaller than~$\lambda_L$ 
exactly if~$R_i < S_{i}$ for some~$i\ge5$. 
This follows directly from
$$R_N = \lambda_L \prod_{i=5}^N \frac{R_i}{R_{i-1}} = \lambda_L \prod_{i=5}^N \frac{R_i}{S_{i}},$$
where we use Lemma~\ref{Mahlerequality} to prove the second equality.
Furthermore, since we have~$S_{i}=R_{i-1}\ge R_i$, it holds that~$R_i\le S_i$ for all~$i\ge5$, and the set of all~$R_N$ accumulates at~$1$ if and only if 
$\prod_{i=5}^N \frac{R_i}{S_{i}}$ converges to~$\lambda_L^{-1}$ as~$N\to\infty$. 
\end{proof}

\subsection{Houses}
Recall that the house of a polynomial is the largest modulus among its roots.
Let~$\lambda_{i}$ and~$\widetilde\lambda_{i}$ be the smallest houses larger than~$1$ among all monic integer reciprocal and 
skew-reciprocal polynomials of degree~$2^i$, respectively.
Furthermore, let~$r_{i} = 2^i\log(\lambda_{i})$ and~$s_{i} = 2^i\log(\widetilde\lambda_{i})$. 

\begin{lem}
\label{inductive_lemma}
For~$i\ge1$, we have $s_{i+1} \ge \mathrm{min}\left\{r_{i}, \log(1.179)\right\}$.
\end{lem}

\begin{proof}
Let~$f\in\mathbf{Z}[t]$ be a monic skew-reciprocal polynomial of degree~$2^{i+1}$ such that~$\house{f}>1$.
We use Lemma~\ref{rekurs} to distinguish two cases. 
Assume for the first case that~$f(t)=g(t^2)$, where~$g(t)$ is a reciprocal polynomial of degree~$2^{i}$.
In this case, we have~$\house{f}^2 = \house{g}$. 
It follows that~$2^{i+1}\mathrm{log}{\house{f}}=2^{i}\mathrm{log}\house{g}\ge r_{i}$. 
On the other hand, if~$f$ has a nonreciprocal irreducible factor that is not~$(t-1)$, then 
Theorem~\ref{breusch} implies~$2^{i+1}\mathrm{log}\house{f} \ge \mathrm{log}(1.179)$. 
\end{proof}

\begin{lem}
\label{Clemma}
The answer to the question of Schinzel and Zassenhaus is positive exactly if the sequence~$\{r_{i}\}$ is bounded strictly away from zero. 
\end{lem}

\begin{proof}
By Theorem~\ref{breusch}, the question of Schinzel and Zassenhaus is equivalent to the same question restricted 
to reciprocal polynomials. Furthermore, any reciprocal polynomial~$f(t)$ can be multiplied by~$(t+1)^k$, 
where~$k$ is at most the degree of~$f(t)$, so that it becomes reciprocal of degree~$2^i$, 
for some~$i\ge1$, keeping its house. This means that the question of Schinzel and Zassenhaus is equivalent 
to the same question for (not necessarily irreducible) reciprocal polynomials of degree~$2^i$. 
The statement of the lemma now follows from the fact that~$\{r_{i}\} = \{2^i\log(\lambda_{i})\}$ 
is strictly bounded away from zero exactly if there exists a constant~$c$ such that~$\lambda_{i} > 1+\frac{c}{2^i}$ for all~$i$. 
\end{proof}

\begin{proof}[Proof of Theorem~\ref{SZthm}]
For one direction, we assume there exists a sequence~$\left\{\frac{q_{N_j}}{q_{n_j}}\right\}$, where~$0<n_j<N_j$, that converges to zero. 
If~$f(t)$ is a reciprocal polynomial of even degree, then~$f(t^2)$ is a skew-reciprocal polynomial. 
This implies~$s_{i} \le r_{i-1}$. 
In particular, we have~$$r_{N_j} = r_{n_j} \prod_{i=n_j+1}^{N_j} \frac{r_{i}}{r_{i-1}}\le 
r_{n_j}\prod_{i=n_j+1}^{N_j} \frac{r_{i}}{s_{i}} \le 4\log(\varphi)\frac{q_{N_j}}{q_{n_j}},$$
where~$\varphi$ is the golden ratio.
For the last inequality, we use~$r_{n_j}\le r_1 = 4\log(\varphi)$. 
The numbers~$r_{N_j}$ converge to~$0$ as~$j\to\infty$, 
giving a negative answer to the question of Schinzel and Zassenhaus by Lemma~\ref{Clemma}. \medskip

\noindent
For the other direction, we assume the set~$\left\{\frac{q_N}{q_n}\in\mathbf{R} : n,N\in\mathbf{N}, n<N \right\}\subset\mathbf{R}$
is bounded away from zero. \medskip

\noindent
\emph{Claim. $r_{N}\ge \frac{\mathrm{log}(1.179)q_N}{\mathrm{max}\left\{q_1,\dots,q_{N-1}\right\}}$.}\medskip

\noindent
We admit the claim for a moment.
By our assumption, there is a constant bounding all fractions~$\frac{q_N}{q_n}$ with~$0<n<N$ away from zero. 
In particular, by the claim, there exists a constant bounding~$r_{N}$ strictly away from zero for all~$N$. 
This is equivalent to a positive answer to the question of Schinzel and Zassenhaus by Lemma~\ref{Clemma}.\medskip

\noindent 
We now prove the claim by induction on~$N$. \medskip

\noindent
\emph{Base case:} For~$N=2$, we verify
$$r_{2} = \frac{r_2}{s_2}\cdot{s_2}
\ge\frac{q_2}{q_1}\mathrm{min}\{r_1,\log(1.179)\}=\frac{\log(1.179)q_2}{\mathrm{max}\{q_1\}},$$
where the inequality is due to Lemma~\ref{inductive_lemma}, and the equality on the right follows from~$r_1=4\log(\varphi)$, 
which is larger than~$\log(1.179)$.\medskip

\noindent
\emph{Inductive step:} We again use Lemma~\ref{inductive_lemma}.
We have~$$r_{N+1} = \frac{r_{N+1}}{s_{N+1}} \cdot s_{N+1}
\ge\frac{q_{N+1}}{q_{N}}\mathrm{min}\left\{r_{N},\mathrm{log}(1.179)\right\}.$$
Using the induction hypothesis on~$r_{N}$, this yields
\begin{align*} r_{N+1} &\ge \frac{q_{N+1}}{q_{N}}\mathrm{min}\left\{ \frac{\mathrm{log}(1.179)q_N}{\mathrm{max}\left\{q_1,\dots,q_{N-1}\right\}},\mathrm{log}(1.179)\right\}\\
&= \mathrm{log}(1.179)q_{N+1}\mathrm{min}\left\{ \frac{1}{\mathrm{max}\left\{q_1,\dots,q_{N-1}\right\} }, \frac{1}{q_N} \right\}\\
&= \frac{\mathrm{log}(1.179)q_{N+1}}{\mathrm{max}\left\{q_1,\dots,q_N\right\}},
\end{align*}
which completes the inductive step.
\end{proof}

\section{Symplectic matrices and mapping classes}
\label{mappingclasses}

\noindent
The goal of this section is to illustrate that monic integer reciprocal and skew-reciprocal polynomials arise naturally in geometry: 
as characteristic polynomials of symplectic and anti-symplectic matrices, respectively. These in turn arise as the actions induced 
on the first homology of closed surfaces by orientation-preserving and orientation-reversing mapping classes, respectively.
\subsection{Symplectic matrices}
An integer matrix~$A$ of size~$2g\times2g$ is \emph{symplectic} if it preserves the standard symplectic form 
$$\Omega=
\begin{pmatrix}
0 & I_g \\
-I_g & 0
\end{pmatrix},$$
that is, if~$A^{\top}\Omega A = \Omega$. An integer matrix~$A$ of size~$2g\times2g$ is \emph{anti-symplectic} if it reverses the standard symplectic form~$\Omega$,
that is, if~$A^{\top}\Omega A = -\Omega$. \medskip

\noindent
The next two lemmas characterise monic integer reciprocal and skew-reciprocal 
polynomials of even degree as the characteristic polynomials of 
integer symplectic and anti-symplectic matrices, respectively. 

\begin{lem}
The characteristic polynomial of an integer symplectic matrix is reciprocal, 
and the characteristic polynomial of an integer anti-symplectic matrix is skew-reciprocal.
\end{lem}

\begin{proof}
The statement for symplectic matrices is a standard fact. An adaptation of the proof to anti-symplectic matrices is given by Strenner and the author~\cite{LS}.
\end{proof}

\begin{lem}
\label{symplecticrep}
Any monic reciprocal polynomial~$f\in\mathbf{Z}[t]$ of even degree is the characteristic polynomial of an integer symplectic matrix,
and any monic skew-reciprocal polynomial~$f\in\mathbf{Z}[t]$ of even degree is the characteristic polynomial of an integer anti-symplectic matrix.
\end{lem}

\begin{proof}
On page 686, Ackermann~\cite{Ackermann} defines a symplectic companion matrix 
$$B=
\begin{pmatrix}
0 & \dots & & & & \dots& 0 & -1 \\
1 & & & & &  & & -a_1 \\
 & \ddots & & & & & & -a_2 \\
 & & \ddots & & & & & \vdots \\
 & & & 1 & -a_1 & -a_2 & \dots & -a_{g} \\
 & & & & \ddots & & & 0 \\
 & & & & & \ddots & & \vdots \\
 & & & & & & 1 & 0 
\end{pmatrix}
$$
for the monic reciprocal polynomial 
$$h(t) = a_gt^g + \sum_{i=0}^{g-1} a_i (t^i + t^{2g-i}),$$
where~$a_0 = 1$. 
Multiplying~$B$ from the right with~$R=
\left(\begin{smallmatrix}
-I_g & 0 \\
0 & I_g
\end{smallmatrix}\right)$
yields an anti-symplectic matrix, since~$$(RB)^{\top}\Omega RB = B^{\top}R\Omega RB=B^{\top}(-\Omega)B = -\Omega.$$
The characteristic polynomial of~$B$ can be calculated by successively developing the first row, until the matrix is of size~$g\times g$,
then successively developing the last column. Multiplication with~$R$ changes only the first~$g$ rows of~$B$, and one can see that it only possibly changes 
the first~$g$ coefficients of the characteristic polynomial. More precisely, one can explicitly check that multiplication with~$R$ 
changes the coefficient of~$t^i$ by a sign~$(-1)^g(-1)^i$, for~$i< g$. Hence, the characteristic polynomial of~$RB$ is the monic skew-reciprocal polynomial
$$ f(t) =  a_gt^g + \sum_{i=0}^{g-1} a_i ((-1)^{g}(-t)^i + t^{2g-i}),$$
where~$a_0=1$. 
Altogether, we get any monic skew-reciprocal polynomial~$f\in\mathbf{Z}[t]$ of even degree as the characteristic polynomial of an
anti-symplectic matrix of the type~$RB$.  
\end{proof}

\subsection{Mapping classes}

Let~$\Sigma_g$ be the orientable closed surface of genus~$g$. 
A \emph{mapping class} of~$\Sigma_g$ is a homeomorphism~$\phi:\Sigma_g\to\Sigma_g$,
up to isotopy. Mapping classes of a fixed surface form a group under composition, the \emph{extended mapping class group}. 

\begin{lem}
The action induced on the first homology~$\mathrm{H}_1(\Sigma_g)$ by an orientation-preserving or 
orientation-reversing mapping class is given by an integer symplectic or anti-symplectic matrix, respectively. 
\end{lem}

\begin{proof}
The action of an orientation-preserving or orientation-reversing mapping class preserves or reverses, respectively, the intersection form on~$\mathrm{H}_1(\Sigma_g)$.
\end{proof}

\begin{lem}
\label{modrep}
Any integer symplectic or anti-symplectic matrix of size~$2g\times 2g$ is obtained as the action induced on the first homology~$\mathrm{H}_1(\Sigma_g)$ 
by some orientation-preserving or orientation-reversing mapping class, respectively.
\end{lem}

\begin{proof}
The statement for orientation-preserving mapping classes and symplectic matrices is a standard fact about the symplectic 
representation of mapping class groups, see, for example,~\cite{Primer}. 
The statement for anti-symplectic matrices follows from the fact that multiplication by an anti-symplectic matrix induces an automorphism of the group 
of symplectic and anti-symplectic matrices. This automorphism sends a symplectic matrix to an anti-symplectic one and vice-versa. 
In particular, by composing mapping classes which
represent all symplectic matrices by an orientation-reversing mapping class, we obtain all anti-symplectic matrices as actions on the first homology.
\end{proof}

\noindent 
The following proposition summarises the results of Section~\ref{mappingclasses} so far.

\begin{prop}
\label{polycharacter}
The monic integer reciprocal and skew-reciprocal 
polynomials of degree~$2g$ are exactly the characteristic polynomials of 
the actions induced on the first homology of the closed surface of genus~$g$ by orientation-preserving 
and orientation-reversing mapping classes, respectively.
\end{prop}

\subsection{Pseudo-Anosov mapping classes}
\label{pAsection}
A mapping class~$f$ of a surface~$\Sigma_g$ is \emph{pseudo-Anosov} 
if there exists a pair of transverse, singular measured~$f$-invariant foliations of~$\Sigma_g$
such that~$f$ stretches one of them by a factor~$\lambda>1$ and the other one by a factor~$\lambda^{-1}$. 
The number~$\lambda$ is called the \emph{dilatation} of~$f$ and is an algebraic integer~\cite{Th}. 
Let~$\delta_g$ be the smallest dilatation among all pseudo-Anosov mapping classes on~$\Sigma_g$. 
Recall that~$\delta_g^\mathrm{hom}$ is the minimal spectral radius larger than~$1$ among 
actions induced on the first homology of the closed surface of genus~$g$ by orientation-preserving mapping classes.\medskip

\noindent
We finish this section by proving Theorem~\ref{dilthm}. 

\begin{proof}[Proof of Theorem~\ref{dilthm}]
By Theorem~\ref{breusch}, the question of Schinzel and Zassenhaus is equivalent to the same question restricted to reciprocal polynomials. 
This is in turn equivalent to the statement for the characteristic polynomials of actions on the first homology induced by orientation-preserving mapping classes, by Proposition~\ref{polycharacter}. 
Thus, a positive answer to the question of Schinzel and Zassenhaus is equivalent to the statement: 
there exists a universal constant~$c'>0$ such that~$\delta_g^\mathrm{hom}$ satisfies~$\delta_g^\mathrm{hom}\ge1+\frac{c}{2g}.$
This is in turn equivalent to the existence of a constant~$c>1$ so that for all~$g$,~$$(\delta_g^\mathrm{hom})^g \ge c.$$

\noindent
For one direction, we assume that this inequality holds. Setting~$C=\log_c(R')$ 
yields~$(\delta_g^\mathrm{hom})^{gC}\ge c^{\log_c(R')} = R' \ge (\delta_g)^g,$
where~$R'$ is the constant from Theorem~\ref{Penner}.
This implies~$(\delta_g^\mathrm{hom})^C\ge\delta_g$. \medskip

\noindent
For the other direction, assume there exists a universal constant~$C$ such that~$$(\delta_g^\mathrm{hom})^C\ge\delta_g.$$
Then, we have~$(\delta_g^\mathrm{hom})^{gC}\ge(\delta_g)^g\ge R,$
where~$R>1$ is the constant in Theorem~\ref{Penner}. In particular, it follows that~$(\delta_g^\mathrm{hom})^g \ge R^{\frac{1}{C}}>1,$
which finishes the proof.
\end{proof}

\end{document}